\documentclass[]{theclass}

\usepackage[symbol]{footmisc}

\begin{document}


\begin{frontmatter}

\titledata{On the existence of graphs which can colour every regular graph}{}

\authordata{Giuseppe Mazzuoccolo}
{Dipartimento di Informatica\\ Universit\`{a} degli Studi di Verona, Italy}{giuseppe.mazzuoccolo@univr.it}{}{}

\authordata{Gloria Tabarelli}{Dipartimento di Matematica\\ Universit\`{a} di Trento,
Italy}{gloria.tabarelli@unitn.it}{}{}

\authordata{Jean Paul Zerafa}
{St. Edward's College, Triq San Dwardu\\ Birgu (Citt\`{a} Vittoriosa), BRG 9039, Cottonera, Malta;\\
Department of Technology and Entrepreneurship Education\\
University of Malta, Malta; \\
Department of Computer Science, Faculty of Mathematics, Physics and Informatics\\ Comenius University, Mlynsk\'{a} Dolina, 842 48 Bratislava, Slovakia}
{zerafa.jp@gmail.com}
{The author was partially supported by VEGA 1/0743/21, VEGA 1/0727/22, and APVV-19-0308.}


\begin{abstract}
Let $H$ and $G$ be graphs. An $H$-colouring of $G$ is a proper edge-colouring $f:E(G)\rightarrow E(H)$ such that for any vertex $u\in V(G)$ there exists a vertex $v\in V(H)$ with $f\left (\partial_Gu\right )=\partial_Hv$, where $\partial_Gu$ and $\partial_Hv$ respectively denote the sets of edges in $G$ and $H$ incident to the vertices $u$ and $v$. If $G$ admits an $H$-colouring we say that $H$ colours $G$. The question whether there exists a graph $H$ that colours every bridgeless cubic graph is addressed directly by the Petersen Colouring Conjecture, which states that the Petersen graph colours every bridgeless cubic graph. In 2012, Mkrtchyan showed that if this conjecture is true, the Petersen graph is the unique connected bridgeless cubic graph $H$ which can colour all bridgeless cubic graphs. In this paper we extend this and show that if we were to remove all degree conditions on $H$, every bridgeless cubic graph $G$ can be coloured substantially only by a unique other graph: the subcubic multigraph $S_{4}$ on four vertices. A few similar results are provided also under weaker assumptions on the graph $G$. In the second part of the paper, we also consider $H$-colourings of regular graphs having degree strictly greater than $3$ and show that: (i) for any $r>3$, there does not exist a connected graph $H$ (possibly containing parallel edges) that colours every $r$-regular multigraph, and (ii) for every $r>1$, there does not exist a connected graph $H$ (possibly containing parallel edges) that colours every $2r$-regular simple graph.
\end{abstract}
\end{frontmatter}

\section{Introduction}\label{section intro}
Graphs considered in this paper are finite, undirected and do not contain any loops. Note that graphs may contain parallel edges, and when we want to emphasise that a graph does or does not admit some parallel edges, we refer to it as a multigraph or a simple graph, respectively. The vertex set and the edge set of a graph $G$ are respectively denoted by $V(G)$ and $E(G)$. Let $U\subseteq V(G)$. The set consisting of all the edges having exactly one endvertex in $U$ is denoted by $\partial_{G}U$, and when it is obvious which graph $G$ we are referring to we just write $\partial U$. When $U$ consists of only one vertex, say $u$, we write $\partial u$, instead of $\partial \{u\}$, for simplicity.
Let $H$ be an arbitrary graph: an \emph{$H$-colouring} of $G$ is a proper edge-colouring $f:E(G)\to E(H)$ of $G$ with edges of $H$, such that for each vertex $u\in V(G)$, there exists a vertex $v \in V(H)$ with $f(\partial_{G}u)=\partial_{H}v$.
If there is no pair of distinct vertices $v$ and $w$ of $H$ such that $\partial_{H}w=\partial_{H}v$, then an $H$-colouring $f$ (of $G$) naturally induces the map $f_V:V(G)\to V(H)$ defined for every vertex $u$ of $V(G)$ as $f_V(u)=v$, where $v$ is the unique vertex of $H$ such that $f(\partial_{G}u)=\partial_{H}v$.
If $G$ admits an $H$-colouring, then we write $H\prec G$ and we say that the graph $H$ colours the graph $G$. Let $P$ denote the well-known Petersen graph. One of the most important conjectures in graph theory is the Petersen Colouring Conjecture by Jaeger. 

\begin{conjecture}[Petersen Colouring Conjecture---Jaeger, 1988 \cite{Jaeger}]\label{conj:petersen}
For any bridgeless cubic graph $G$, $P\prec G$.
\end{conjecture}

Conjecture \ref{conj:petersen} implies several other relevant conjectures in the field of graph theory such as the Berge--Fulkerson Conjecture \cite{BergeFulkerson} (see also \cite{MazzuoccoloEquivalence}). Weaker conjectures on bridgeless cubic graphs implied by the Berge--Fulkerson Conjecture are the Fan--Raspaud Conjecture \cite{FanRaspaud} (see also \cite{MacajovaSkovieraOddness2}), and the $S_{4}$-Conjecture \cite{MazzuoccoloS4} which states the following\footnote[3]{During the revision process of this paper, Conjecture \ref{conj:s4} was proved to be true by Kardo\v{s}, M\'{a}\v{c}ajov\'{a} and the last author  (see \cite{KardosMacajovaZerafa}).}.

\begin{conjecture}[$S_{4}$-Conjecture---Mazzuoccolo, 2013 \cite{MazzuoccoloS4}]\label{conj:s4}
For every bridgeless cubic graph $G$, there exist two perfect matchings such that the deletion of their union leaves a bipartite subgraph of $G$.
\end{conjecture}
We remark that in \cite{gmjp}, the first and last author showed that Conjecture \ref{conj:s4} is equivalent to saying that for every bridgeless cubic graph $G$, $S_{4}\prec G$, where $S_{4}$ is the subcubic multigraph portrayed in Figure \ref{figure s4}.

\begin{figure}[h]
\begin{subfigure}[t]{.32\textwidth}
\centering
      \includegraphics[width=0.4\textwidth]{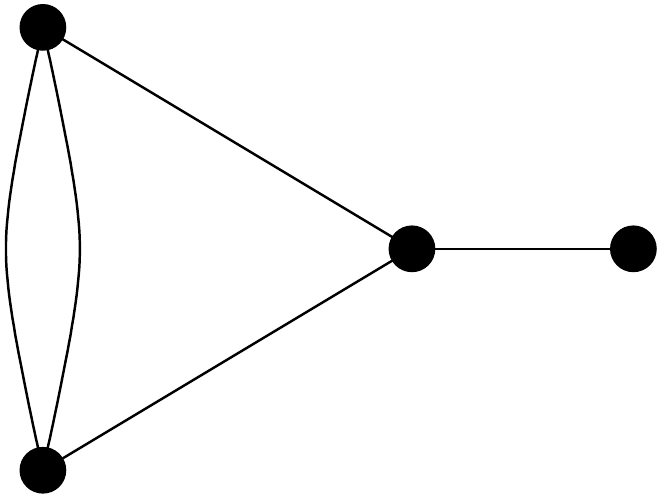}
      \caption{The multigraph $S_4$}
      \label{figure s4}
\end{subfigure}
\begin{subfigure}[t]{.32\textwidth}
  \centering
  \includegraphics[width=.7\textwidth]{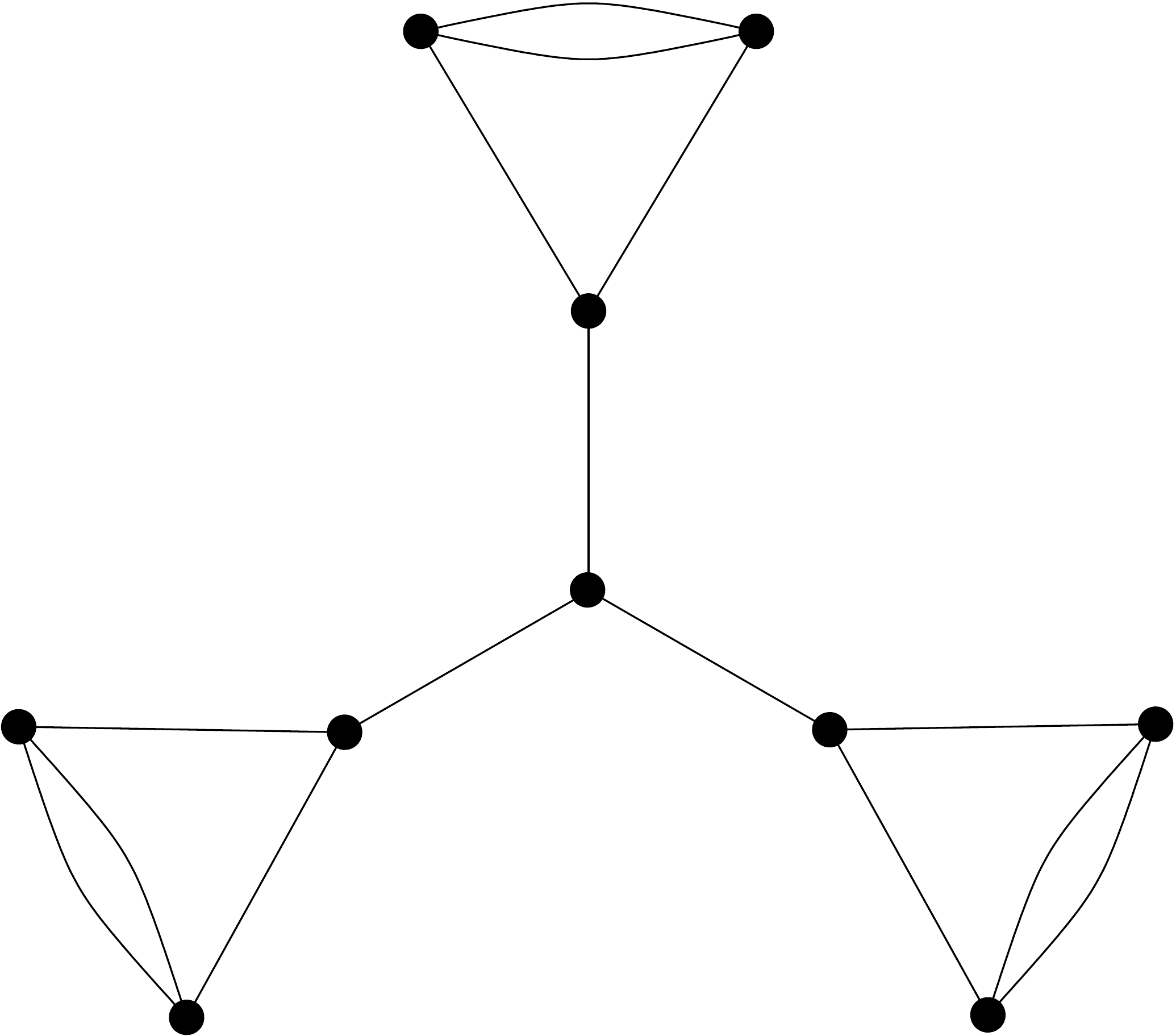}
\caption{The Sylvester graph $S_{10}$}
      \label{figure s10}
\end{subfigure}
\begin{subfigure}[t]{.32\textwidth}
  \centering
  \includegraphics[width=.7\textwidth]{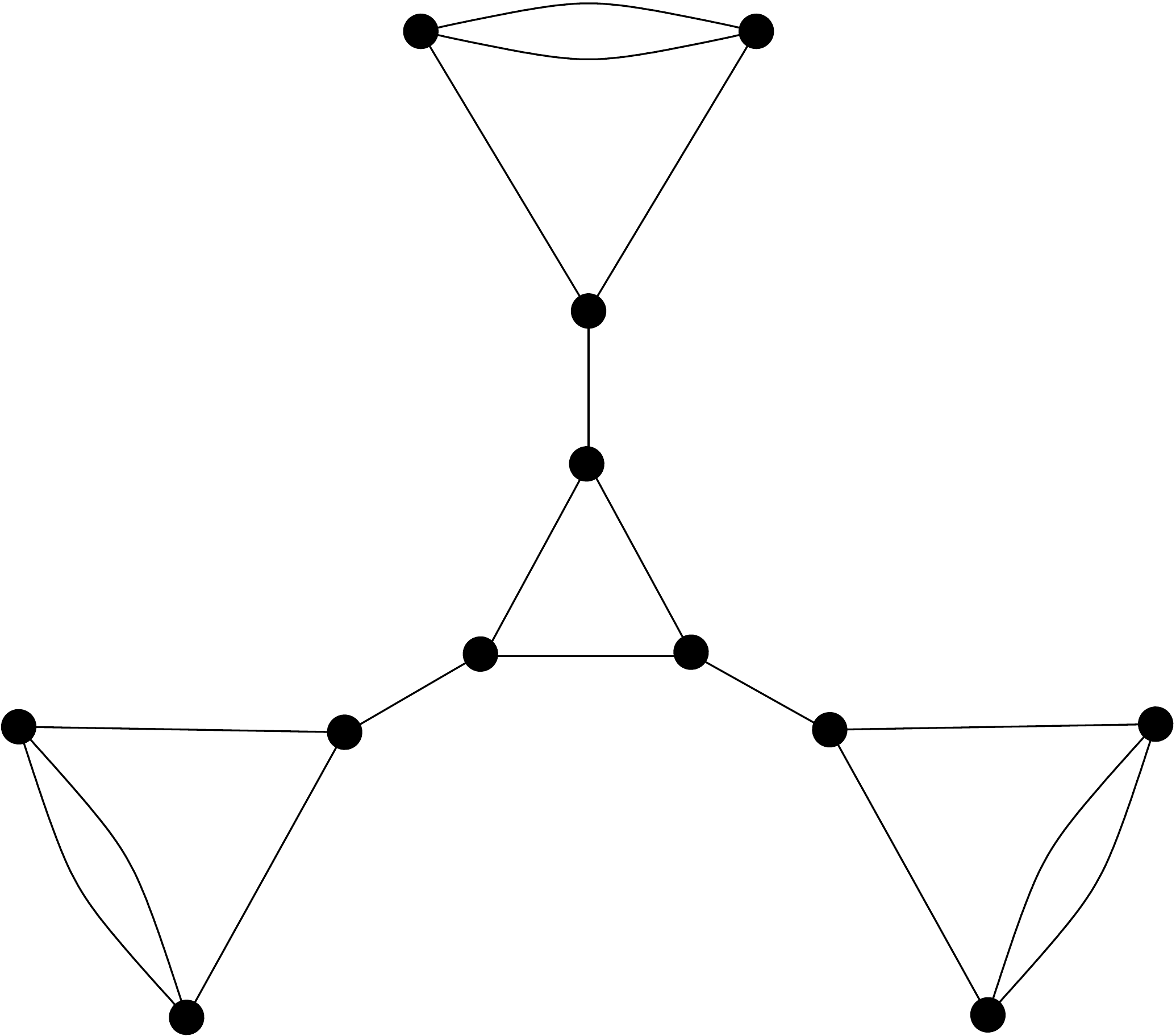}
\caption{The multigraph $S_{12}$}
\label{figure s12}
\end{subfigure}

\caption{}
\label{figure s4 s10 s12}
\end{figure}

Proving the Petersen Colouring Conjecture would also confirm the Cycle Double Cover Conjecture \cite{seymour, szekeres, ZhangBook} which is a conjecture stated for general graphs and not only for cubic graphs. It is due to these huge consequences that the Petersen Colouring Conjecture is, arguably, one of the most trying and arduous conjectures in graph theory. In the same spirit of Jaeger's Conjecture, Mkrtchyan also proposed the following two conjectures for cubic graphs, for which connectivity conditions are relaxed---in fact, the following two conjectures are stated for cubic graphs which are not necessarily bridgeless.

\begin{conjecture}[$S_{12}$-Conjecture---Mkrtchyan, 2012 \cite{VahanPetersen}]
\label{conj:S12}
For each cubic graph $G$ admitting a perfect matching, $S_{12}\prec G$. 
\end{conjecture}

\begin{conjecture}[$S_{10}$-Conjecture---Mkrtchyan, 2012 \cite{VahanPetersen}]
\label{conj:S10}
For each cubic graph $G$, $S_{10}\prec G$. 
\end{conjecture}

The multigraph $S_{10}$ is also referred to as the \emph{Sylvester graph} and is depicted together with the multigraph $S_{12}$ in Figure \ref{figure s4 s10 s12} (see also \cite{sylvester}). 

Mkrtchyan proved the following theorem (Theorem 2.4 in \cite{VahanPetersen}).

\begin{theorem}[Mkrtchyan, 2012 \cite{VahanPetersen}]\label{theorem Vahan Petersen}
If $H$ is a connected bridgeless cubic graph with $H\prec P$, then $H \simeq P$.
\end{theorem}

Consequently, the following holds.

\begin{corollary}[Mkrtchyan, 2012 \cite{VahanPetersen}]\label{corollary Vahan Petersen}
If $H$ is a connected bridgeless cubic graph such that $H\prec G$ for every bridgeless cubic graph $G$, then $H\simeq P$. 
\end{corollary}

In other words, the previous result says that we cannot replace the Petersen graph in Conjecture \ref{conj:petersen} with any other connected bridgeless cubic graph.
Nevertheless, if we choose $H$ from the larger class of connected cubic graphs (not necessarily bridgeless), there are other possible candidates. 
In particular, if we minimise the assumptions on the graph $H$ by considering the class of connected graphs (not even cubic), then another candidate is given by the graph $S_4$.

Theorem \ref{thm:main} is one of the main results of this paper, and it is a generalisation of Theorem \ref{theorem Vahan Petersen}: it is obtained by removing any restriction on the degree of the vertices of the graph $H$ in an $H$-colouring of the Petersen graph. Analogously, Corollary \ref{cor:main} is the natural generalisation of Corollary \ref{corollary Vahan Petersen}, but, in order to explain its statement, we need to introduce the following terminology. Let $G$ be a multigraph having three degree $3$ vertices and a further vertex of arbitrary degree. Denote this set of four vertices by $X$. If the induced multisubgraph $G[X]$ is isomorphic to $S_4$, then we say that $G$ \emph{exposes} $S_4$ and that $G[X]$ is an \emph{exposed copy} of $S_{4}$ in $G$. Observe that both $S_{10}$ and $S_{12}$ expose (three times) $S_4$.

Indeed, as a consequence of Theorem \ref{thm:main} we prove that the unique graphs that can colour every bridgeless cubic graph are exactly $P$ and all graphs which expose $S_4$ (see Corollary \ref{cor:main}).
In a similar way, Corollary \ref{cor1:S10} and Corollary \ref{cor1:S12} in Section \ref{section cubic} would follow if Conjecture \ref{conj:S10} and Conjecture \ref{conj:S12} are respectively true.

All the above mentioned conjectures deal with the question asking whether there exists a connected graph $H$ such that $H \prec G$ for any $G$ in a given class of cubic graphs. Table \ref{Table_cubic_case} shows the possibilities for the eventual existence of such a graph $H$, and is divided according to the cases when $H$ is assumed to be a simple graph or a graph with parallel edges. In this table, we consider three classes of graphs (that may admit parallel edges) to be coloured by some connected graph $H$: (i) bridgeless cubic graphs, (ii) cubic graphs admitting a perfect matching, and (iii) cubic graphs. By Corollary \ref{cor:main}, Corollary \ref{cor:S10} and Corollary \ref{cor:S12}, if the graph $H$ that colours all the graphs in each of the corresponding classes exists, then the only possibilities are the ones presented in the table. 

\begin{remark}\label{remark table1}
Theorem \ref{thm:main}, together with the fact that it is possible to construct cubic graphs with a perfect matching having a subgraph as in Figure \ref{fig2_12}, implies that a connected simple graph $H$ that colours any cubic graph $G$ with a perfect matching does not exist (see \cite{MazzuoccoloMkrtchyan} for details). Even more so, there is no connected simple graph that colours any cubic graph.
\end{remark}

\begin{figure}[h]
      \centering
      \includegraphics[width=0.25\textwidth]{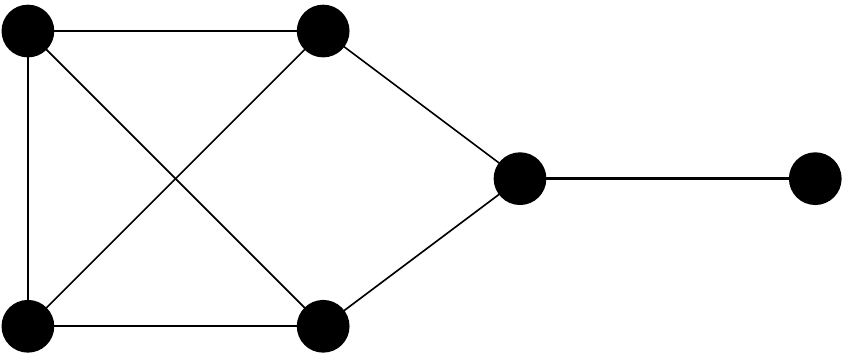}
      \caption{The subgraph mentioned in Remark \ref{remark table1}}
      \label{fig2_12}
\end{figure}

\begin{table}[h]
\centering
\begin{tabular}{cccc}
\emph{Cubic graphs}  & \emph{$H$ simple graph} & \emph{$H$ with parallel edges} \\
\cmidrule[1pt]{1-3}
 bridgeless & $H\simeq P$ (Theorem \ref{thm:main})  & $H_f\simeq S_4$ (Theorem \ref{thm:main}) \\
\cmidrule[0.5pt]{1-3}
with a perfect matching & $\nexists$ (Remark \ref{remark table1})  & $H \simeq S_{10}$ or $H\simeq S_{12}$ (Corollary  \ref{cor:S12})
\\
\cmidrule[0.5pt]{1-3}
any & $\nexists$ (Remark \ref{remark table1}) & $H\simeq S_{10}$ (Corollary \ref{cor:S10})
\\
\cmidrule[1pt]{1-3}
\end{tabular}
\caption{Possibilities for the eventual existence of an $H$-colouring for different classes of cubic graphs}\label{Table_cubic_case}
\end{table}

In the second part of the paper we partially answer the question dealing with whether there exists a graph $H$ such that $H \prec G$ for any $r$-regular graph $G$, for $r>3$, in a given class. The results obtained are summarised in Table \ref{Table_reg_case}.

\begin{table}[h]
\centering
\begin{tabular}{cc}
$r$-regular graphs, $r>3$  & $H$ (multi)graph \\
\cmidrule[1pt]{1-2}
simple graphs  & $\nexists$ for $r$ even (Theorem \ref{theorem G simple even degree}) \\
\cmidrule[0.5pt]{1-2}
multigraphs & $\nexists$ for any $r$ (Theorem \ref{theorem r_regular})
\\
\cmidrule[1pt]{1-2}
\end{tabular}
\caption{Non existence of an $H$-colouring for $r$-regular simple graphs and multigraphs}
\label{Table_reg_case}
\end{table}

\section{Notation and technical lemmas}


Before continuing, we need some further definitions and notation which we introduce in order to focus our study only on the relevant part of $H$ in a given $H$-colouring $f$ of some graph $G$. In what follows, the irrelevant part of $H$ shall arise due to the vertices $v\in V(H)$ for which $v\not\in\textrm{Im}(f_{V})$. Such vertices may occur in $H$, and in the sequel they shall be referred to as unused.

\begin{lemma}\label{lem:connected}
Let $G$ be a connected graph and let $f:E(G)\to E(H)$ be an $H$-colouring of $G$. Then, the induced subgraph $H[\textrm{Im}(f)]$ of $H$ is connected. 
\end{lemma}
\begin{proof}
Observe that by definition of $H$-colouring, if $e_1$ and $e_2$ are two adjacent edges of $G$, then $f(e_1)$ is adjacent to $f(e_2)$ in $H[\textrm{Im}(f)]$. The result follows immediately by the connectivity assumption on $G$.
\end{proof}

By the previous lemma, from now on we can assume that $H$ is connected, since only the edges of one connected component belong to the image of any $H$-colouring of a connected graph $G$.
Note that if $H$ is connected then the map $f_V$ is well defined for any given $H$-colouring $f$, except if $H$ is the graph $tK_2$ on two vertices and with $t$ parallel edges between them.  Moreover, it is straightforward that a graph $G$ admits a $tK_2$-colouring if and only if $G$ is $t$-regular and $t$-edge-colourable and consequently, if and only if it admits a $K_{1,t}$-colouring, where $K_{1,t}$ is the star on $t+1$ vertices. Hence, it is not restrictive assuming $|V(H)|>2$ in what follows.

Let $H$ and $G$ be connected graphs such that $H \prec G$ and $|V(H)|>2$. Let $f$ be an $H$-colouring of $G$ and consider the map $f_V$. We denote by $H_f$, the edge-induced subgraph $H[\textrm{Im}(f)]$ and with a slight abuse of terminology we shall refer to the graph $H_f$ as the image of the $H$-colouring $f$. Note that in general $\textrm{Im}(f_V) \subseteq V(H_f)$, since an edge $uv$ of $H_f$ must have at least one of its endvertices $u$ and $v$ in $\textrm{Im}(f_V)$, but not necessarily both of them. Every vertex of $H_f$ which does not belong to $\textrm{Im}(f_V)$ is said to be \emph{unused}.

Starting from the graph $H_f$, we can obtain a large variety of connected graphs, say $H'$, such that $G$ admits an $H'$-colouring. A first easy procedure is obtained by considering an arbitrary connected graph $H'$ having $H_f$ as an induced subgraph with the further property that $d_{H'}(v)=d_{H_f}(v)$ for every $v \in \textrm{Im}(f_V)$. A more general way is obtained by  eventually splitting in advance unused vertices of $H_f$ in arbitrary graphs (see Figure \ref{figure notation} for a possible example, where splitting of vertices is also portrayed).  Finally, we remark that if $H_f$ has no unused vertex (that is, $H_f=H$), then no connected graph $H'$ different from $H$ can be obtained as a combination of previous operations.

\begin{definition}
Let $G$ and $H$ be connected graphs such that $|V(H)|>2$ and $H \prec G$. Let $f$ be an $H$-colouring of $G$ and let $f_V$ be the induced map on the vertices of $G$. We define the graph $\tilde{H}_f$ as the graph obtained from $H_f$ by splitting every unused vertex $u$ of $H_f$ into $d_{H_f}(u)$ vertices of degree $1$. We refer to the graph $\tilde{H}_f$ as the \emph{splitted image} of $f$.
\label{def:htilde} 
\end{definition}

\begin{figure}[h]
      \centering
      \includegraphics[width=0.7\textwidth]{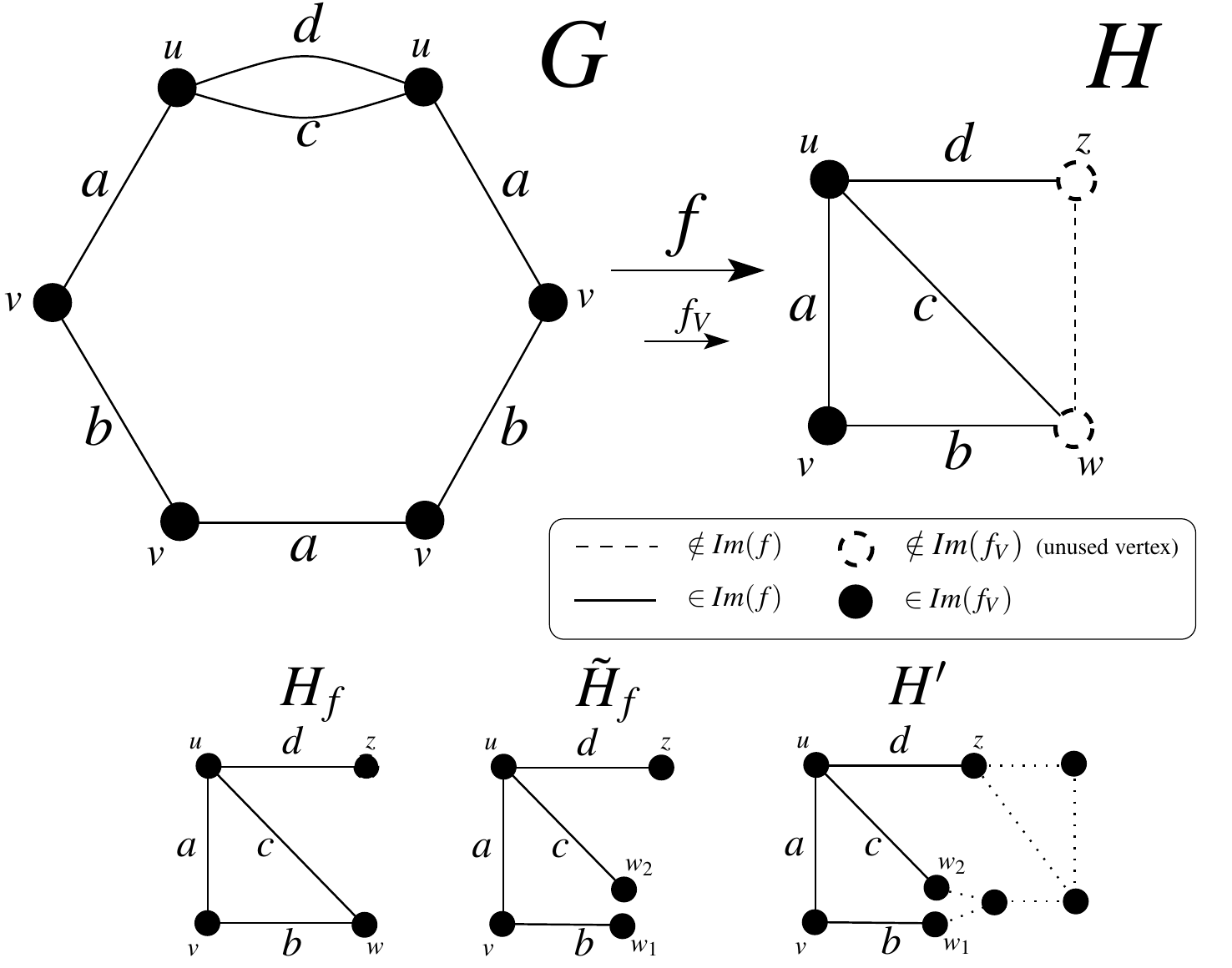}
      \caption{$H$, $H_f$, $\tilde{H}_f$ and a possible example for $H'$}
      \label{figure notation}
\end{figure}

In what follows, with a slight abuse of notation, we shall always refer in the same way to a vertex $u$ in $\textrm{Im}(f_V)$ independently to whether we are considering it in $H$, $H_f$ or $\tilde{H}_f$. For simplicity, the functions corresponding to an $H_f$-colouring and an $\tilde{H}_f$-colouring of some graph $G$ are both denoted by $f$ as well. An unused vertex $v$ is referred to in the same way both in $H$ and in $H_{f}$, whilst the vertices of $\tilde{H}_{f}$ obtained by splitting $v$ are referred to as \emph{the vertices arising from $v$}. Finally, we remark that since $G$ is connected, every two distinct vertices in $H_{f}$ are the endvertices of a path whose inner vertices all belong to $\textrm{Im}(f_{V})$. Consequently, $\tilde{H}_{f}$ is connected by Lemma \ref{lem:connected}.

In what follows we make use of some results contained in Lemma 2.2 in \cite{HakMkr}. We reproduce only the part of the lemma that we shall need in the sequel, even if in a slightly more general form. Moreover, we add and prove statement \emph{(d)}.

\begin{lemma}
Let $G$ and $H$ be graphs with $H \prec G$, and let $f$ be an $H$-colouring of $G$.
\begin{enumerate}[(a)]
 \item If $M$ is any matching of $H$, then $f^{-1}(M)$ is a matching of $G$.
 \item $\chi'(G)\leq \chi'(H)$ (where $\chi'$ denotes the chromatic index of a graph).
 \item If $M$ is a perfect matching of $H$, then $f^{-1}(M)$ is a perfect matching of $G$.
\item If $G$ is connected, let $X$ be an edge-cut of $H_f$ such that $H_{f} - X$ does not contain any isolated vertex. Then $f^{-1}(X)$ is an edge-cut of $G$.
\end{enumerate}
\label{lem:properties}
\end{lemma}
\begin{proof}
Statements \emph{(a)}, \emph{(b)}, \emph{(c)} follow from \cite{HakMkr}, so it suffices to prove statement \emph{(d)}.

\emph{(d)} Let $\overline{G}=G - f^{-1}(X)$ and $\overline{H}=H_f - X$. Consider $\overline{f} \colon E(\overline{G}) \to E(\overline{H})$, the restriction of $f$ to $\overline{G}$. Since $X=f(f^{-1}(X))$, the function $\overline{f}$ is an $\overline{H}$-colouring of $\overline{G}$, and since $H_{f}- X$ does not contain any isolated vertex, it holds that ${\overline{H}}_{\overline{f}}=\overline{H}$. Suppose that $f^{-1}(X)$ is not an edge-cut of $G$, for contradiction. This means that $\overline{G}$ is connected. However, by Lemma \ref{lem:connected}, ${\overline{H}}_{\overline{f}}=\overline{H}$ is connected, contradicting $X$ being an edge-cut of $H_f$. 
\end{proof}

Before we continue, we prove the following lemma which gives statement \emph{(c)} of Lemma \ref{lem:properties} as a corollary. This lemma shall also be used in Section \ref{section regular simple}.

\begin{lemma}\label{lemma matching pm}
Let $G$ and $H$ be graphs with $|V(H)|>2$. Let $f$ be an $H$-colouring of $G$ and let $f_{V}$ be the induced map on the vertices of $G$. If $M$ is a matching of $H$ such that every vertex $v\in\textrm{Im}(f_{V})$ is matched in $M$, then $f^{-1}(M)$ is a perfect matching of $G$.
\end{lemma}
\begin{proof}
Since $M$ is a matching of $H$, by Lemma \ref{lem:properties}, $f^{-1}(M)$ is a matching of $G$, so it suffices to show that $f^{-1}(M)$ covers all the vertices of $G$. For each $u\in V(G)$, $f_{V}(u)\in\textrm{Im}(f_{V})$, and so there exists a unique edge $e\in M$ such that $e$ is incident to the vertex $f_{V}(u)$ in $H$. This means that for every vertex $u\in V(G)$, there exists exactly one edge in $\partial_{G}u$ which is coloured by an edge in $M$, implying that $f^{-1}(M)$ is a perfect matching of $G$, as required.
\end{proof}

\section{\emph{H}-colourings of cubic graphs}\label{section cubic}

Before proving the main result of this section (Theorem \ref{thm:main}) we need some further technical results for the case when $G$ is cubic.

\begin{remark}\label{remark:1_3degrees}
Consider an $H$-colouring $f$ of a connected cubic graph $G$. For every vertex $u \in V(\tilde{H}_f)$ exactly one of the following holds:
\begin{itemize}
 \item $u$ has degree $1$ in $\tilde{H}_f$ and either it is itself an unused vertex in $H_f$ or it arises from an unused vertex of $H_f$; or
 \item $u$ has degree $3$ in $\tilde{H}_f$ and it is a vertex of $H$ which belongs to $\textrm{Im}(f_V)$.
\end{itemize}
\end{remark}

\begin{lemma}\label{lemma exactly 1 comp}
Let $H$ be a connected graph. Let $f$ be an $H$-colouring of the Petersen graph $P$. If $e=uv$ is a bridge in $\tilde{H}_f$, then exactly one of $u$ and $v$ has degree $1$ in $\tilde{H}_f$. 
\end{lemma}

\begin{proof}
Let $f_V$ be the map induced by $f$ on the vertices of $P$ and, for contradiction, suppose that both $u$ and $v$ belong to $\textrm{Im}(f_V)$, which results in both vertices having degree $3$ in $\tilde{H}_f$, by Remark \ref{remark:1_3degrees}. Hence, all edges in $\partial_{H} u$ (and $\partial_{H} v$) belong to $\textrm{Im}(f)$, that is they belong to the edge-set of $\tilde{H}_f$. In particular, the edge $e=uv$ belongs to $\textrm{Im}(f)$. Let $\l_{1}$ and $\l_{2}$ be the two edges incident to $u$ in $\tilde{H}_f$ other than $uv$, and let $r_{1}$ and $r_{2}$ be the other two edges incident to $v$ in $\tilde{H}_f$. Since $e$ is an edge-cut and a matching of $\tilde{H}_f$, by Lemma \ref{lem:properties}, $f^{-1}(e)$ is an edge-cut and a matching of $P$. The only matchings of the Petersen graph  which are also edge-cuts are perfect matchings of $P$. Consequently, $f^{-1}(e)$ is a perfect matching of $P$, say $M$, which can be chosen arbitrarily due to the symmetry of the Petersen graph (in a more precise terminology we remark that the Petersen graph is 3-arc-transitive, see for example \cite{pet arc trans}).
The complement of $M$ in the Petersen graph consists of two disjoint $5$-cycles. 
Without loss of generality, by following the notation used in Figure \ref{figure lemma exactly 1 comp}, we can assume that:
\begin{enumerate}[(i)]
\item each edge $u_iv_i$ has colour $e$;
\item $f_V(u_i)=u$, for every vertex $u_i$ of the outer $5$-cycle; and
\item $f_V(v_i)=v$, for every vertex $v_i$ of the inner $5$-cycle.
\end{enumerate}    

It follows that all the edges in the outer $5$-cycle (similarly, inner $5$-cycle) should be alternately mapped to $\l_1$ and $\l_2$ (respectively, $r_{1}$ and $r_{2}$) by $f$. However, this is not possible since these two cycles have odd length. Hence, since $e \in \textrm{Im}(f)$ implies that at least one of $u$ and $v$ belongs to $\textrm{Im}(f_V)$, by Remark \ref{remark:1_3degrees} we conclude that exactly one of the vertices $u$ and $v$ belongs to $\textrm{Im}(f_V)$ and, the one which does not, has degree $1$. 
\end{proof}

\begin{figure}[h]
      \centering
      \includegraphics[width=0.6\textwidth]{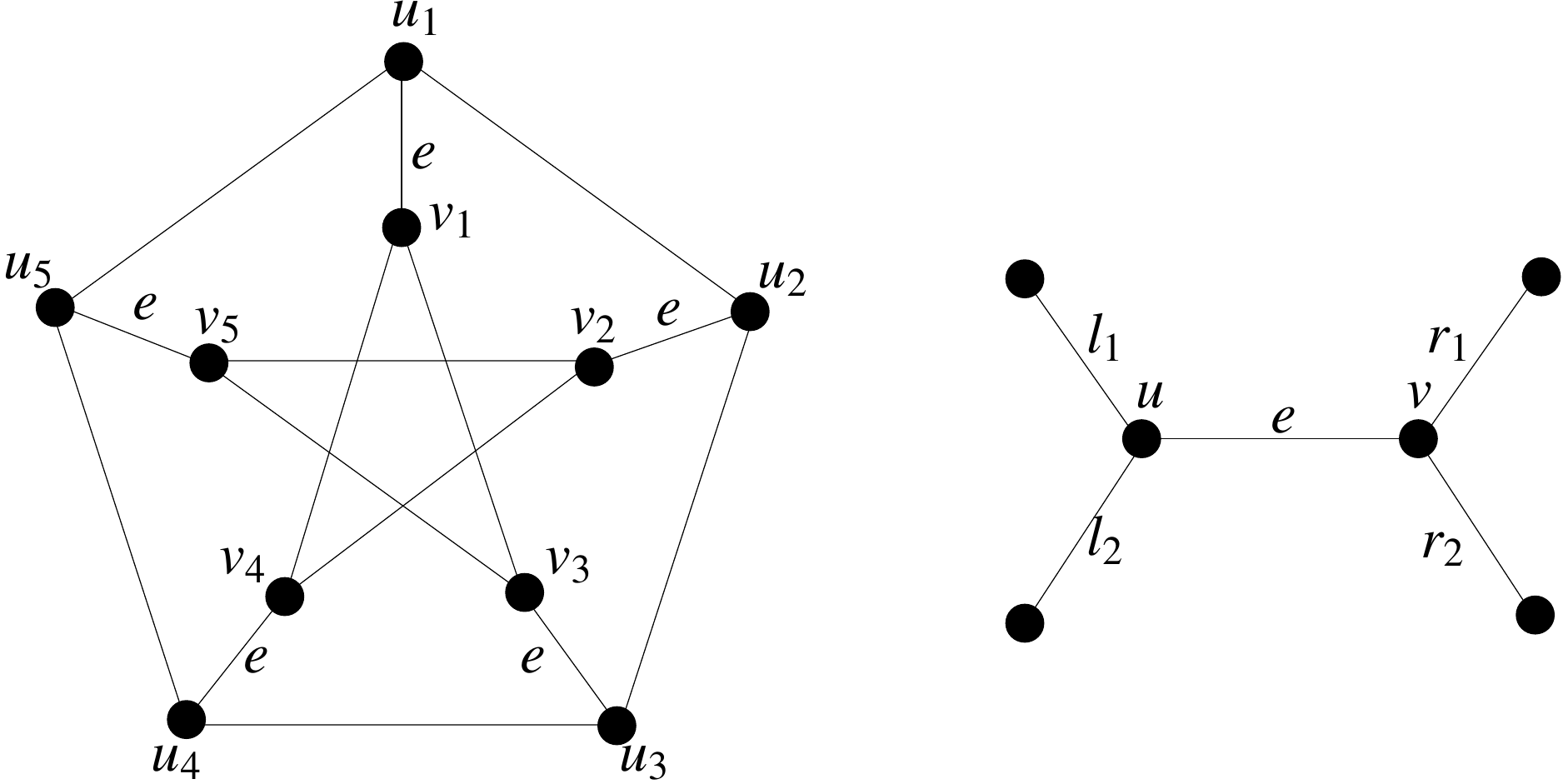}
      \caption{Steps from Lemma \ref{lemma exactly 1 comp}}
      \label{figure lemma exactly 1 comp}
\end{figure}

\begin{theorem}\label{thm:main}
Let $H$ be a connected graph such that $H\prec P$ and let $f$ be an $H$-colouring of $P$. Then, either $H=H_f\simeq P$ or $H_f\simeq S_{4}$.
\end{theorem}

\begin{proof}
First, assume that $\tilde{H}_f$ is cubic. By Remark \ref{remark:1_3degrees}, $\tilde{H}_f=H_f$, and by Lemma \ref{lemma exactly 1 comp}, it follows that it is bridgeless, and so, by Theorem \ref{theorem Vahan Petersen}, $H=H_f \simeq P$. 
Hence, we can assume that $\tilde{H}_f$ is not cubic, and so, by Remark \ref{remark:1_3degrees}, it admits a vertex $v$ whose degree is 1 and is adjacent to a vertex $u$ whose degree is $3$. 
Let $e=uv$, and let the other two edges in $\tilde{H}_f$ incident to $u$ be denoted by $a$ and $b$. By Lemma \ref{lemma exactly 1 comp}, the edges $a$ and $b$ cannot share a further endvertex other than $u$, otherwise $\tilde{H}_f$ is isomorphic to the connected $3$-edge-colourable graph on four vertices with two degree 1 vertices and two degree 3 vertices, thus implying that the Petersen graph $P$ is $3$-edge-colourable by Lemma \ref{lem:properties}, a contradiction. Therefore, $a$ and $b$ share exactly one endvertex ($u$), and we let $w$ and $z$ be the two distinct vertices in $V(\tilde{H}_f) \setminus \{v\}$ such that $a=uw$ and $b=uz$. 

Without loss of generality, assume that the spoke $u_{1}v_{1}$ of $P$ is coloured by $e=uv$. Since $v$ does not belong to $\textrm{Im}(f_V)$, $f_{V}(u_{1})=f_{V}(v_{1})=u$, and so we can assume further that $f(u_{1}u_{5})=f(v_{1}v_{4})=a$ and $f(u_{1}u_{2})=f(v_{1}v_{3})=b$, as in Figure \ref{figure proof theorem1 1}. The case when $f(u_{1}u_{5})=f(v_{1}v_{3})=a$ and $f(u_{1}u_{2})=f(v_{1}v_{4})=b$ is equivalent by the symmetry of $P$. Since $P$ is not 3-edge-colourable, $u$ cannot be the only vertex in $\textrm{Im}(f_V)$. Hence, at least one of $w$ and $z$ must also have degree equal to 3 in $\tilde{H}_f$. By Lemma \ref{lemma exactly 1 comp}, since $\tilde{H}_{f}$ cannot admit a bridge with both of its endvertices having degree 3, the vertices $w$ and $z$ must both have degree $3$ in $\tilde{H}_{f}$.\\
\begin{figure}[h]
      \centering
      \includegraphics[width=0.6\textwidth]{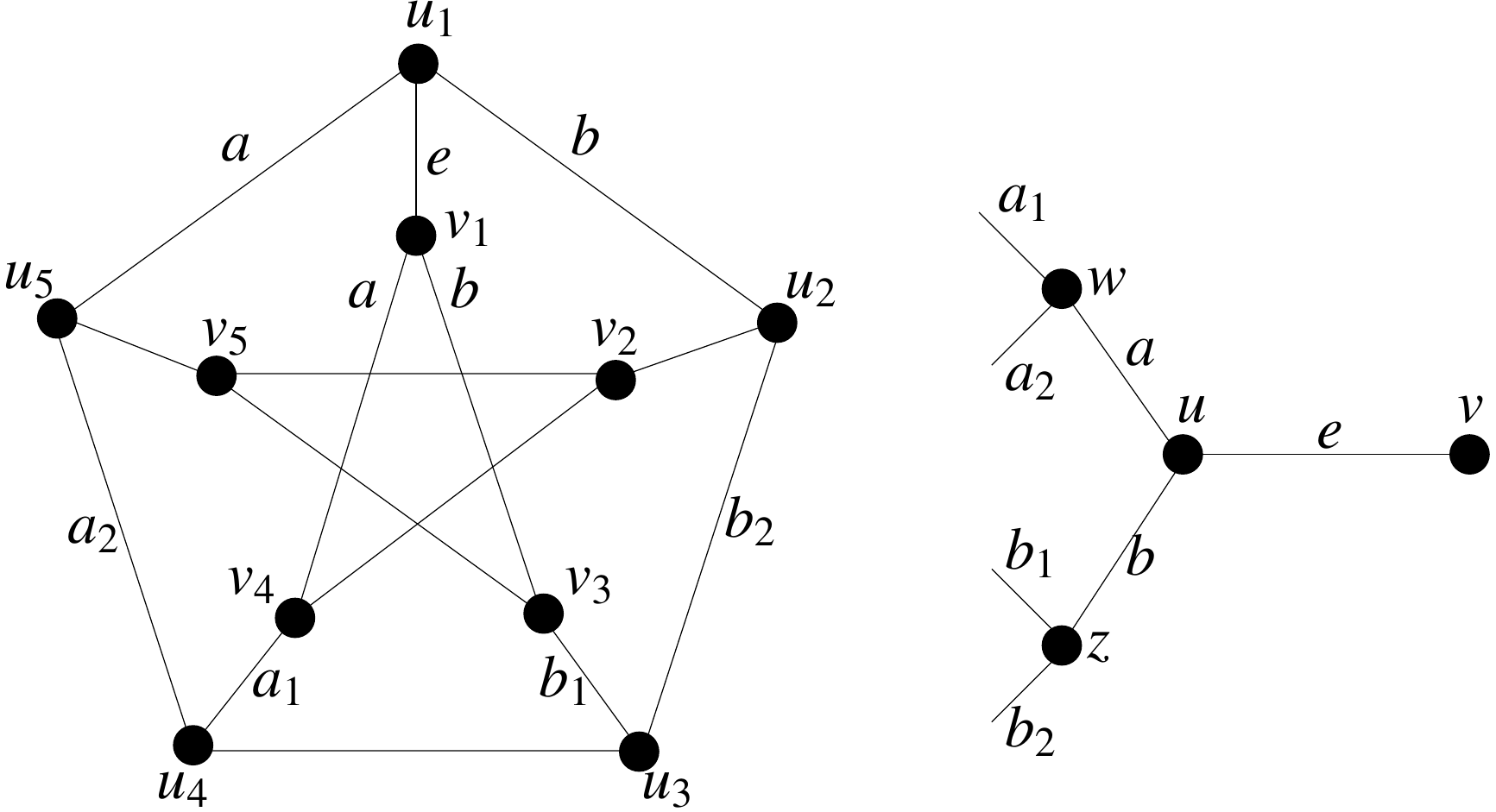}
      \caption{The edges coloured $a_{1},a_{2},b_{1},b_{2}$ in $P$}
      \label{figure proof theorem1 1}
\end{figure}

\noindent\textbf{Claim A.} $u_1v_1$ is the unique edge with colour $e$.

\noindent\emph{Proof of Claim A.} 
Suppose there is another edge $m$ in $P$ which is coloured by $e$. Either $m$ is at distance 1 from $u_1v_1$ or it is at distance 2. Then, there exists $C$, a $5$-cycle or a $6$-cycle, respectively, of $P$ passing through both $u_1v_1$ and $m$. Hence, the other edges in $C$ are coloured by $a$ or $b$, since they are all incident with some edge which is coloured by $e$. 

If $C$ is a $6$-cycle, the vertices of $C$ that are not incident with $u_1v_1$ or $m$, have two of their incident edges coloured by $a$ and $b$, implying that they are mapped by $f_V$ into $u$, and thus their third incident edges, say $l_1$ and $l_2$ respectively, are also coloured by $e$. This is a contradiction, since $l_1$ and $l_2$ are edges of $P$ incident to a common vertex. 
If $C$ is a $5$-cycle, there exists a vertex of $C$ having two of its incident edges coloured by $a$ and $b$, implying that it is mapped by $f_V$ into $u$ and thus its third incident edge is also coloured by $e$. But in this case, there exists a $5$-cycle $C'$ of $P$ whose edges are incident to some edge coloured by $e$, implying that all edges of $C'$ must be coloured by $a$ and $b$. This is a contradiction since $C'$ is an odd cycle.\,\,\,{\tiny$\blacksquare$}
\\

Let $a_{1}$ and $a_{2}$ be the two edges in $\tilde{H}_f - a$ which are incident to the vertex $w$, and let $b_{1}$ and $b_{2}$ be the two edges in $\tilde{H}_f- b$ which are incident to the vertex $z$. Since no edge but $u_1v_1$ has colour $e$ in $P$, all the edges incident to an edge with colour $a$ (similarly, $b$) receive colours $a_{1}$ and $a_{2}$ (respectively, $b_{1}$ and $b_{2}$). Hence, without loss of generality we can assume:
\begin{enumerate}[(i)]
\item $f(u_{4}v_{4})=a_{1}$ and $f(u_{4}u_{5})=a_{2}$; and
\item $f(u_{3}v_{3})=b_{1}$ and $f(u_{2}u_{3})=b_{2}$,
\end{enumerate}
as in Figure \ref{figure proof theorem1 1}. \\

\noindent\textbf{Claim B.} $a_{1}=b_{1}$ and $a_{2}=b_{2}$.

\noindent\emph{Proof of Claim B.} Due to the edge $u_{3}u_{4}$ in $P$, there exists an edge $g$ in $\tilde{H}_f$ such that $a_1,a_2,g$ are incident with a common vertex, and $b_1,b_2,g$ are incident with a common vertex. Moreover, since $f(v_{1}v_{4})=f(u_{1}u_{5})=a$, we have $f(v_{2}v_{4})=a_{2}$ and $f(u_{5}v_{5})=a_{1}$. Similarly, $f(v_{3}v_{5})=b_{2}$ and $f(u_{2}v_{2})=b_{1}$. This means that $a_{1}$ and $b_{2}$ share a common vertex in $\tilde{H}_f$, and similarly, $a_{2}$ and $b_{1}$ share a common vertex in $\tilde{H}_f$. Since the vertices of $\tilde{H}_{f}$ can have degree 1 and 3, the only way how the above statements can be satisfied is by having $\{a_{1},a_{2}\}=\{b_{1},b_{2}\}$. In particular, since $f(u_{5}v_{5})=a_{1}$ and $f(v_{3}v_{5})=b_{2}$, $b_{2}$ must be equal to $a_{2}$, proving our claim.\,\,\,{\tiny$\blacksquare$}\\

By Claim B,  $\tilde{H}_f \simeq S_{4}$, and since $\tilde{H}_f$ has a unique vertex of degree 1, it cannot be obtained by splitting unused vertices of some other graph, and so, $\tilde{H}_f=H_f$, as required (in Figure \ref{figure S4 P} an $S_4$-colouring of $P$ is represented).
\end{proof}
\begin{figure}[h]
      \centering
      \includegraphics[width=0.6\textwidth]{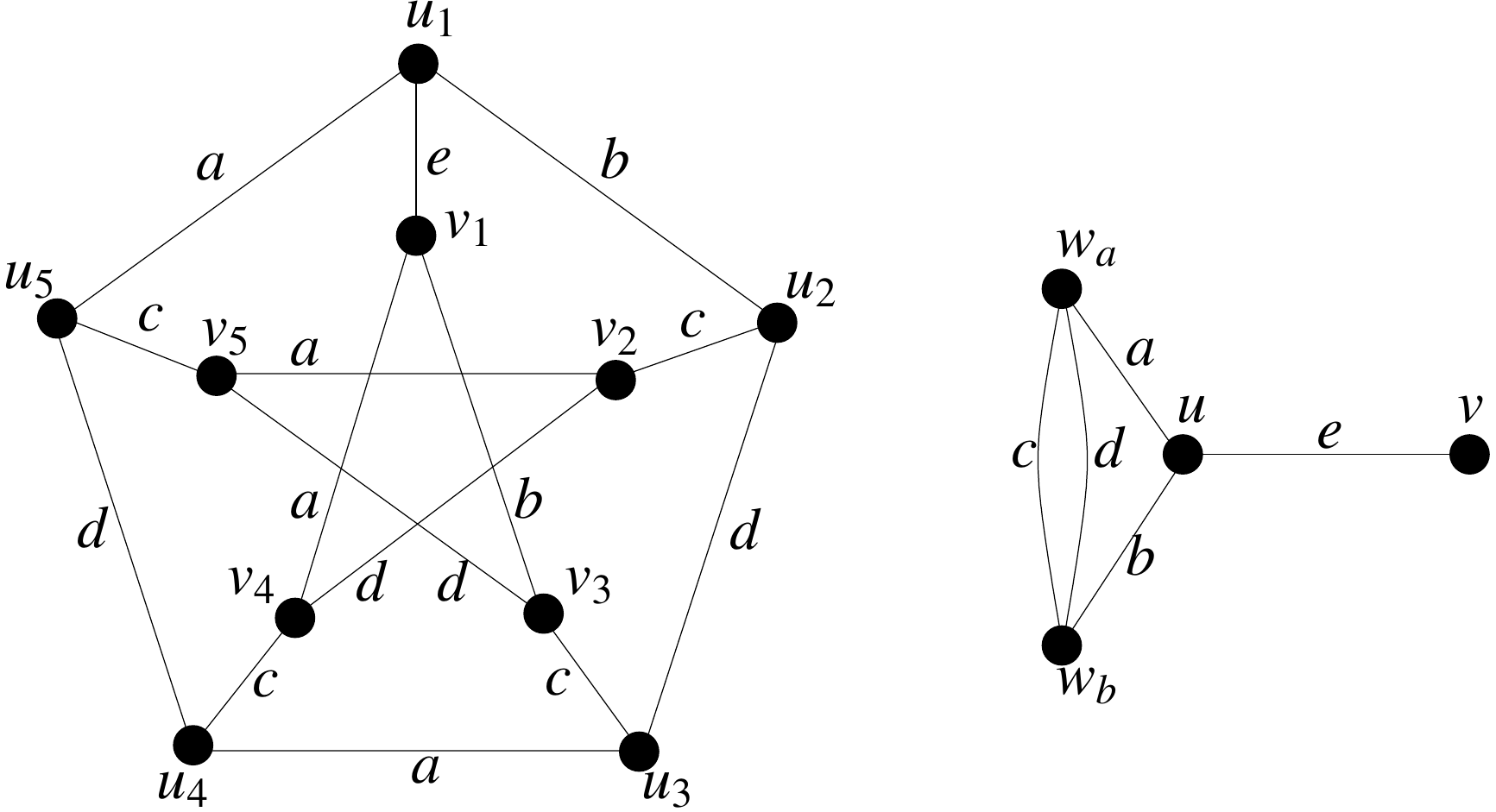}
      \caption{An $S_{4}$-colouring of $P$}
      \label{figure S4 P}
\end{figure}

As before, since the Petersen graph is bridgeless and cubic, the following holds.

\begin{corollary}\label{cor:main}
If there exists a connected graph $H$ colouring all bridgeless cubic graphs, then either $H \simeq P$ or $H$ exposes $S_{4}$.
\end{corollary}

To conclude this section we provide a generalisation of the following two theorems, proved in \cite{VahanPetersen} and \cite{HakMkr}.

\begin{theorem}[Mkrtchyan, 2013 \cite{VahanPetersen}]
 Let $H$ be a connected cubic graph with $H \prec S_{10}$. Then $H \simeq S_{10}$.
 \label{thm:MkrtchyanS10}
\end{theorem}

\begin{theorem}[Hakobyan \& Mkrtchyan, 2019 \cite{HakMkr}]
Let $H$ be a connected cubic graph with $H \prec S_{12}$. Then, either $H \simeq 
S_{10}$ or $H \simeq  S_{12}$.
\label{thm:MkrtchyanS12}
\end{theorem}

More specifically, in the same way as Theorem \ref{thm:main} generalises Theorem \ref{theorem Vahan Petersen}, the next corollaries generalise previous results by removing the regularity assumption on the graph $H$. 

\begin{cor}\label{cor1:S10}
Let $H$ be a connected graph with $H \prec S_{10}$. Then, $H\simeq S_{10}$.
\end{cor}
\begin{cor}\label{cor1:S12}
Let $H$ be a connected graph with $H \prec S_{12}$. Then, either $H\simeq S_{10}$ or $H\simeq S_{12}$.
\end{cor}

Both these corollaries are a direct consequence of Theorem \ref{thm:MkrtchyanS10} and Theorem \ref{thm:MkrtchyanS12}. Indeed, let $H$ a connected graph and let $f$ be an $H$-colouring of $S_{10}$ (similarly, $S_{12}$). Suppose $\tilde{H}_f$ is not cubic: then $\tilde{H}_f$ can be extended to infinitely many
connected cubic graphs by the procedure described just above Definition \ref{def:htilde}.  All of them colour $S_{10}$ (respectively, $S_{12}$), a contradiction to Theorem \ref{thm:MkrtchyanS10} (respectively, Theorem \ref{thm:MkrtchyanS12}). Hence, $\tilde{H}_f$ is cubic and the statements respectively follow by Theorem \ref{thm:MkrtchyanS10} and Theorem \ref{thm:MkrtchyanS12}, once again.

As before, once we recall that $S_{12}$ has a perfect matching, two other corollaries follow from Corollary \ref{cor1:S10} and Corollary \ref{cor1:S12}.

\begin{corollary}\label{cor:S10}
If there exists a connected graph $H$ colouring all cubic graphs, then $H\simeq S_{10}$.
\end{corollary}

\begin{corollary}\label{cor:S12}
If there exists a connected graph $H$ colouring all cubic graphs with a perfect matching, then either $H\simeq S_{10}$ or $H\simeq S_{12}$.
\end{corollary}

\section{\emph{H}-colourings in \emph{r}-regular graphs, for \emph{r$>$3}}\label{section regular}

In this section we analyse whether there exists a connected graph $H$ such that every $r$-regular graph $G$ admits an $H$-colouring, for each $r>3$. Clearly, the answer could depend on the class of graphs from where we choose the graph $G$: the bigger the class, the more unlikely it is that the same graph $H$ would colour all of them. 

In Section \ref{section regular multigraphs} we consider the case of $G$ admitting parallel edges. On the other hand, in Section \ref{section regular simple} we restrict our attention to the subclass of simple regular graphs. In the former case, we are able to give a complete negative answer, whilst in the latter one we give a negative answer for $G$ having even degree, and we leave the odd case as an open problem (see Problem \ref{problem odd regular}).

\subsection{\emph{H}-colourings in \emph{r}-regular multigraphs, for \emph{r$>$3}}\label{section regular multigraphs}

In this section we show that, for every even $r>3$, there is no connected graph $H$ such that $H \prec G$ for every $r$-regular multigraph $G$. We note that $H$ is not necessarily simple and can contain parallel edges, that is, $H$ is a graph in the general sense as explained in Section \ref{section intro}.

In each of the multigraphs $S_{4}, S_{6}$ and $S_{12}$, portrayed in Figure \ref{figure s4s6s12}, there is a unique way how one can pair all the vertices of each multigraph such that the vertices in each pair are adjacent. Consequently, these three multigraphs each admit a unique perfect matching up to which parallel edges are chosen, shown in bold in Figure \ref{figure s4s6s12}. Notwithstanding whether we are referring to $S_{4}, S_{6}$ or $S_{12}$, in Section \ref{section regular multigraphs}, we shall refer to this perfect matching in each of these multigraphs by $M$. For every $k\geq 0$, let $S_4+kM$ (similarly, $S_{6}+kM$ or $S_{12}+kM$) be the $(k+3)$-regular multigraph obtained from $S_4$ (respectively, $S_{6}$ or $S_{12}$) after adding $k$ edges parallel to every edge in $M$. When $k=0$, $S_4+0M$, $S_6+0M$ and $S_{12}+0M$ are assumed to be $S_4,S_6$ and $S_{12}$, respectively.

\begin{figure}[h]
      \centering
      \includegraphics[width=0.45\textwidth]{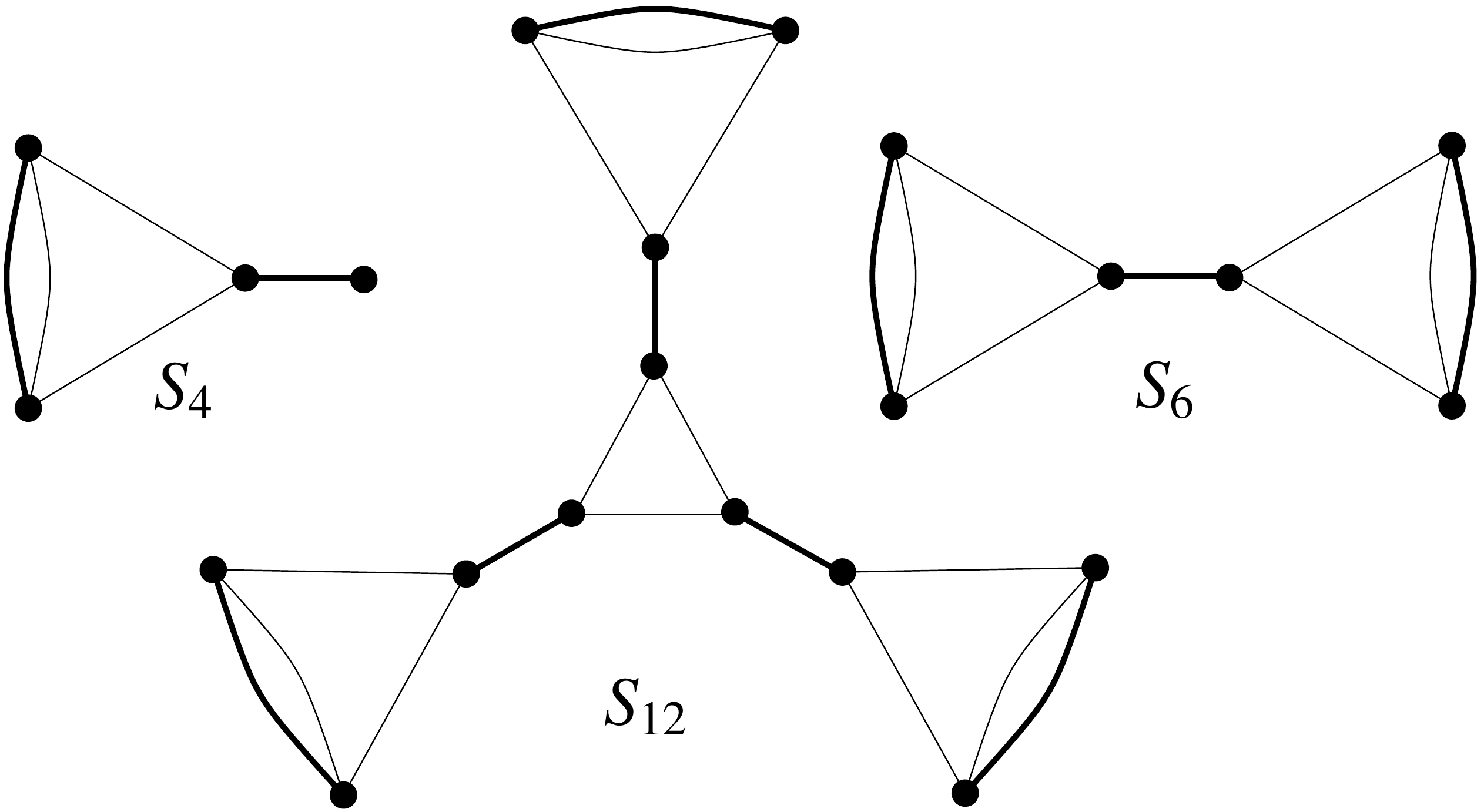}
      \caption{The chosen perfect matching $M$ for the multigraphs $S_{4}, S_{6}$ and $S_{12}$ }
      \label{figure s4s6s12}
\end{figure}

In analogy with the already introduced definition of an \emph{exposed copy of $S_4$} we define in detail what an \emph{exposed copy of $S_4+kM$} is, for some $k \geq 0$. Let $G$ be a multigraph having three vertices of degree $k+3$ and a further vertex of arbitrary degree.  
Denote this set of vertices by $X$. If the induced multisubgraph $G[X]$ is isomorphic to $S_4+kM$, then we say that $G$ \emph{exposes} $S_4+kM$ and that $G[X]$ is an \emph{exposed copy} of $S_{4}+kM$ in $G$.


 

In the next proposition we show that for any $r>3$ there exists an $r$-regular multigraph $G$ that admits only $G$-colourings.

\begin{proposition}\label{prop:S12+kM}
Let $H$ be a connected graph with $H \prec S_{12}+kM$, for some $k \geq 1$. Then, $H\simeq S_{12}+kM$. 
\end{proposition}
\begin{proof}

\begin{figure}[h]
      \centering
      \includegraphics[scale=0.5]{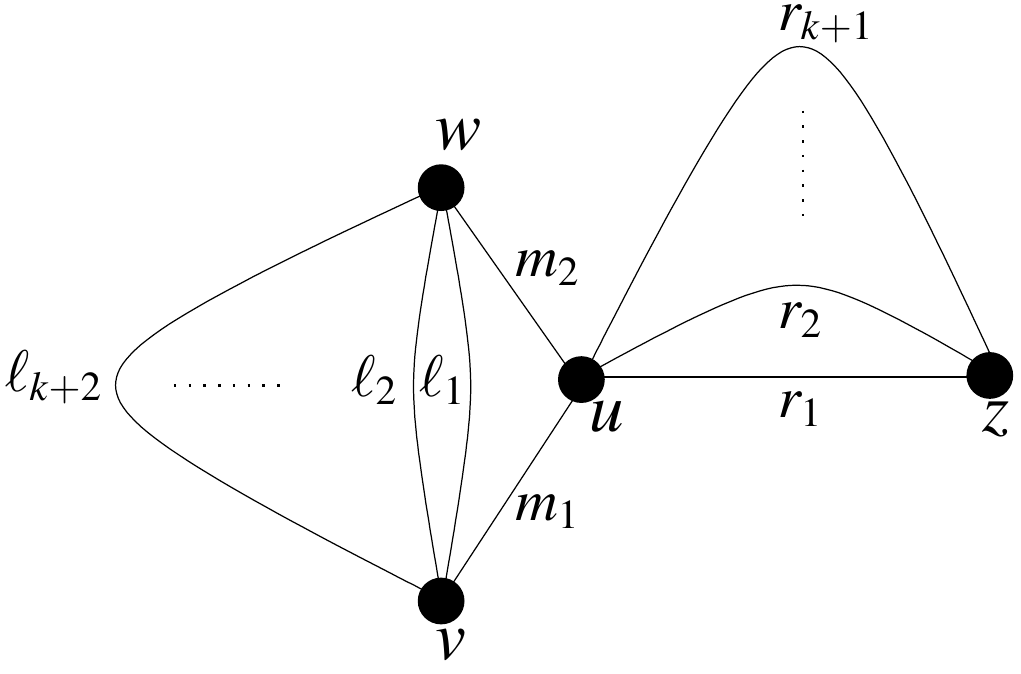}
      \caption{General labelling of an exposed copy of $S_{4}+kM$ in $G$ }
      \label{S_4+kM-labels}
\end{figure}
Let $f$ be an $H$-colouring of $G$, where $G=S_{12}+kM$. Let $z^1, z^2, z^3$ be the three vertices of $G$ which induce a 3-cycle (without parallel edges). Let $Z_{1}, Z_{2}, Z_{3}$ be the three disjoint exposed copies of $S_4+kM$ in $G$, such that $z^{i}\in Z_{i}$, for each $i\in\{1,2,3\}$. Additionally, for each $i \in \{1,2,3\}$, we label the remaining vertices of $Z_i$ by $u^i,v^i$ and $w^i$, where $u^i$ is the unique vertex adjacent to $z^i$ in $Z_i$. The $k+2$ edges with endvertices $v^i$ and $w^i$ are labeled by $\ell_1^i,...,\ell_{k+2}^i$, whilst the $k+1$ edges with endvertices $u^i$ and $z^i$ by $r_1^i,...,r_{k+1}^i$. Finally, the two edges $u^iv^i$ and $u^iw^i$ are denoted by $m^i_1$ and $m^i_2$, respectively. In what follows, when we refer to a generic exposed copy of $S_4+kM$ in $G$ we will omit the superscripts in the labelling of vertices and edges of $G$ (see Figure \ref{S_4+kM-labels}), and in their images under the action of $f_V$.
 
We first show that for each exposed copy of $S_{4}+kM$ in $G$, the following holds. By the definition of $H$-colouring, $f(\ell_1), f(\ell_2), \ldots, f(\ell_{k+2}), f(m_{1}), f(m_{2})$ are $k+4$ distinct edges in $H$ since the edges $\ell_1, \ell_2, \ldots, \ell_{k+2}, m_{1}, m_{2}$ are distinct and pairwise adjacent in $G$.
Hence, $f_V(v)\neq f_V(w)$. Indeed, if by contradiction $f_V(v)$ and $f_V(w)$ are equal, say to $x\in V(H)$, then all the edges $f(\ell_1), f(\ell_2), \ldots, f(\ell_{k+2}), f(m_{1}), f(m_{2})$ are incident to $x$ since $\ell_1, \ell_2, \ldots, \ell_{k+2}, m_{1}, m_{2}$ are exactly all the edges incident to $v$ and $w$ in $G$. These add up to $k+4$ edges, meaning that $d_H(x)=k+4$. However, $d_G(v)=k+3$, that is, $d_H(f_V(v))\neq d_G(v)$, a contradiction. 
It follows that $f(\ell_1), f(\ell_2), \ldots, f(\ell_{k+2})$ are parallel edges in $H$ with endvertices $f_V(v)=v'$ and $f_V(w)=w'$. In particular, $f(m_{1})$ must be incident to $v'$ and $f(m_{2})$ must be incident to $w'$. Moreover, since $m_{1}$ and $m_{2}$ are adjacent edges in $G$, $f(m_{1})$ and $f(m_{2})$ are adjacent edges in $H$. Denote by $u'\in V(H)-\{v',w'\}$ their common endvertex. Consequently, $f_V(u)=u'$, and the edges $f(r_{1}), \ldots, f(r_{k+1})$ are incident to $u'$ in $H$ but not to $v'$ and $w'$.


We now prove that $f_V(z^i)\neq f_V(u^i)$ for each $i \in \{1,2,3\}$. Without loss of generality, suppose that $f_V(z^1)=f_V(u^1)=u'^1$, for contradiction. Since  $r_1^1,...r_{k+1}^1$ are all incident with both $m_1^1$ and $m_2^1$, it must be that $\{f(z^1z^2), f(z^1z^3)\}=\{f(m_1^1),f(m_2^1)\}$, implying that $ f(z^2z^3) \in \{f(\ell_1^1), \ldots, f(\ell^1_{k+2}) \}$.
Let $e$ be an edge in $\{f(\ell_1^1),\ldots,f(\ell^1_{k+2})\}\setminus\{f(z_2z_3)\}$. Then, $\{e,f(m_1^1),f(m_2^1)\}$ induces a cycle in $H$. Hence, the preimage of such a set induces a $2$-regular subgraph in $G$ that contains $z^1z^2$, $z^1z^3$ but not $z^2z^3$, a contradiction.
This follows because if $f'$ is an $H'$-colouring of a graph $G$ and $F$ is a $k$-regular subgraph of $H'$ containing at least one vertex of $\textrm{Im}(f'_V)$, then $f'^{-1}(E(F))$ induces a $k$-regular subgraph of $G$.
%
Consequently, $f_V(z^1)\neq f_V(u^1)$, as required.\\

Hence, up to now we have proved that the induced multisubgraph $H[f_V(V(Z_i))]$ is an exposed copy of $S_4+kM$ in $H$. From now on, we denote $H[f_V(V(Z_i))]$ by $Z_i'$, for each $i \in \{1,2,3\}$.\\

\noindent\textbf{Claim A.} $Z_{1}', Z_{2}', Z_{3}'$ are pairwise edge-disjoint.\\  
\noindent\emph{Proof of Claim A.} Without loss of generality, suppose that $E(Z_{1}')\cap E(Z_{2}')\neq \emptyset$, for contradiction. Since $f_V$ maps the vertices of $Z_{i}$ having degree $k+3$  into vertices of degree $k+3$ in $H$, either  $Z_{1}'=Z_{2}'$, or $H\left[ (E(Z_{1}')\cup E(Z_{2}')\right]\simeq S_6+kM$.
First, assume that $Z_{1}'=Z_{2}'$. In this case, $f_V(z^1)=f_V(z^2)$, and, without loss of generality, we assume that $f(r^1_j)=f(r^2_j)$ for every $j \in \{1,2,\ldots, k+1\}$. Moreover, all the edges $f(z^1z^2)$, $f(z^2z^3)$ and $f(z^1z^3)$ must be pairwise distinct in $H$, and each of them must be incident to $f_V(z^1)$ (which is equal to $f_V(z^2)$). None of the edges $f(z^1z^2)$, $f(z^2z^3)$ and $f(z^1z^3)$ coincide with $f(r^1_j)$ (which is equal to $f(r^2_j)$) for any $j \in \{1,2,\ldots,k+1\}$, since $z^1z^2$, $z^1z^3$ and $z^2z^3$ are all incident to at least one of $r^1_j$ and $r^2_j$ in $G$. However, this means that $d_H(f_V(z^1))>k+3$, a contradiction. Therefore, we must have the other case, that is, $H\left[E(Z_{1})\cup E(Z_{2})\right]\simeq S_6+kM$. However, since $H$ is connected, if $H\left[E(Z_{1})\cup E(Z_{2})\right]\simeq S_6+kM$, then $H\simeq S_6+kM$, meaning that either $Z_{1}'=Z_{3}'$ or $Z_{2}'=Z_{3}'$, a contradiction once again.\,\,\,{\tiny$\blacksquare$}\\

Hence, $H$ contains three edge-disjoint exposed copies of $S_4+kM
$. Let $W=\{z^1z^2,\linebreak z^2z^3,z^1z^3\}$. Observe that $f(z^1z^2)$, $f(z^2z^3)$ and $f(z^1z^3)$ are pairwise distinct and pairwise adjacent in $H$, so that the possibilities for the edge-induced subgraph $H[f(W)]$ by the edges of $W$ in $H$ are: a 3-cycle ($H[f(W)]\simeq C_3$), a single vertex of degree 3, say $z'$, which is adjacent to three distinct neighbours ($H[f(W)]\simeq K_{1,3}$), or a single vertex of degree 3, say $z'$, having two distinct neighbours.\\

\noindent\textbf{Claim B.} The only possibility for $H[f(W)]$ is a 3-cycle, that is, $H[f(W)] \simeq C_3$.\\
\noindent\emph{Proof of Claim B.} Indeed, in both the other cases there are at least two pairs of edges, say, the pair $\{f(z^1z^2), f(z^2z^3)\}$ and the pair $\{f(z^1z^2), f(z^1z^3)\}$, such that the unique endvertex of the edges in each pair is the vertex $z'$ in $H$. This means that $z^1$ and $z^2$ are mapped into $z'$. Since $Z_1'$ and $Z_2'$ are edge-disjoint in $H$ and the edge $f(z^1z^2)$ must be incident to all the edges $\{f(r^i_{1}), \ldots, f(r^i_{k+1}):i=1,2\}$ of $Z_1'$ and $Z_2'$ , the vertex $z'$, which belongs to $\textrm{Im}(f_{V})$, has degree at least $2(k+1)+1$, a contradiction, since $G$ is $(k+3)$-regular and $2k+3>k+3$ for $k \geq 1$. \,\,\,{\tiny$\blacksquare$}\\ 

Moreover, for any $j \in \{1,2,\ldots,k+1\}$, $f(r^{1}_{j})$ must be incident to $f(z^1z^2)$ and $f(z^1z^3)$, $f(r^{2}_{j})$ with $f(z^1z^2)$ and $f(z^2z^3)$, and, $f(r^{3}_{j})$ with $f(z^2z^3)$ and $f(z^1z^3)$. Combining Claim A and Claim B with these last necessary properties we deduce that $H$ is isomorphic to $G$.
\end{proof}

We are now in a position to prove the main result of this section. 

\begin{theorem}\label{theorem poorly matchable graphs}
For each $r>3$, there is no connected graph $H$ colouring all $r$-regular multigraphs admitting a perfect matching.
\end{theorem}

\begin{proof}
Suppose such a graph $H$ exists. For each fixed $r>3$, choose $G=S_{12}+(r-3)M$, where $M$ is the perfect matching in $S_{12}$ as in Figure \ref{figure s4s6s12}. Since $S_{12}+(r-3)M$ is $r$-regular and $H \prec S_{12}+(r-3)M$, by Proposition \ref{prop:S12+kM} we have that $H$ must be $S_{12}+(r-3)M$.
Now, let $G_r$ be an $r$-regular multigraph admitting an $\left(S_{12}+(r-3)M\right)$-colouring $f$. Since $r>3$, $S_{12}+(r-3)M$ contains (at least) two disjoint perfect matchings, say $M_1$ and $M_2$, and consequently, $f^{-1}(M_1)$ and $f^{-1}(M_2)$ are two disjoint perfect matchings of $G_r$, by Lemma \ref{lem:properties}.
Hence, in order to find an $r$-regular multigraph with a perfect matching and without an $\left(S_{12}+(r-3)M\right )$-colouring, it suffices to exhibit an $r$-regular multigraph admitting a perfect matching but without two disjoint perfect matchings, for every $r>3$. Examples of such multigraphs are constructed in \cite{Rizzi} and called poorly matchable (see also \cite{gm}). The assertion follows.
\end{proof}

In the previous theorem, we consider $G$ as an $r$-regular multigraph admitting a perfect matching. Clearly, the result holds in the larger class of $r$-regular multigraphs.

\begin{theorem}\label{theorem r_regular}
For each $r>3$, there is no  connected graph $H$ colouring all $r$-regular multigraphs.
\end{theorem}

\subsection{\emph{H}-colourings in \emph{r}-regular simple graphs, for \emph{r$>$3}}\label{section regular simple}

In this section our aim is to show that, for every even $r>3$, there is no connected graph $H$ such that $H \prec G$ for every simple $r$-regular graph $G$. We remark that $H$ is not necessarily simple and can contain parallel edges, as in the previous section.

Before proceeding, let $\mathcal{K}_{t}^{r}$ denote the family of $r$-regular multigraphs of order $t$, whose vertices are pairwise adjacent. Note that a graph $G$ in $\mathcal{K}_{t}^{r}$ admits a $t$-clique as a spanning (simple) subgraph of $G$. 

\begin{lemma}\label{lemma k2r+1}
Let $H$ be a connected graph. For every $r\geq 1$, if the complete graph $K_{2r+1}$ admits an $H$-colouring, then  $H\in\mathcal{K}_{t}^{2r}$, where $t$ is an odd integer and no vertex of $H$ is unused.
\end{lemma}

\begin{proof}
Let $f:E(K_{2r+1})\rightarrow E(H)$ be an $H$-colouring of $K_{2r+1}$, and let $f_{V}$ be the induced map on the vertices of $K_{2r+1}$. Let $v_{1}$ and $v_{2}$ be two distinct vertices in $\textrm{Im}(f_{V})$. Note that these two vertices exist since $|\textrm{Im}(f_V)|=1$ would imply that $K_{2r+1}$ is $2r$-edge-colourable. We claim that $v_{1}v_{2}\in E(H)$. Let $u_{1}$ and $u_{2}$ be two (distinct) vertices in $V(K_{2r+1})$ such that $f_{V}(u_{i})=v_{i}$, for each $i\in\{1,2\}$. Since $u_{1}$ is adjacent to $u_{2}$, $f(u_{1}u_{2})$ is incident to both $v_{1}$ and $v_{2}$, implying that $v_{1}v_{2}\in E(H)$. This proves our claim.

Consequently, there exists an integer $t\in\{2, 3,,\ldots,2r+1\}$ such that $H$ contains a complete graph $K_{t}$ as a subgraph and whose vertex set is $\textrm{Im}(f_{V})\subseteq V(H)$. For simplicity, we shall refer to this subgraph as $K_{t}$. Next, we claim that $t$ must be odd. For, suppose not, and assume that $t$ is even. Let $M$ be a matching of $H$ that is also a perfect matching of $K_{t}$. Consequently, $M$ covers all the vertices of $\textrm{Im}(f_{V})$, since $V(K_{t})=\textrm{Im}(f_{V})$. However, by Lemma \ref{lemma matching pm}, $f^{-1}(M)$ is a perfect matching of $K_{2r+1}$, a contradiction, since $K_{2r+1}$ does not admit a perfect matching. Therefore, $t$ must be odd. 

We next claim that $H$ contains a simple spanning subgraph isomorphic to a $t$-clique, that is, $\textrm{Im}(f_{V})=V(H)$. For, suppose not. Then, there exists an edge $xy\in E(H)$, such that $x\in\textrm{Im}(f_{V})$ and $y\not\in\textrm{Im}(f_{V})$. Let $M'$ be a matching of $H$ with $|M'|=\frac{t-1}{2}$ such that $M'$ covers all the vertices of $\textrm{Im}(f_{V})$ except $x$. Let $N=M'\cup \{xy\}$. The set of edges $N$ is a matching of $H$ which covers all the vertices in $\textrm{Im}(f_{V})$. However, by Lemma \ref{lemma matching pm}, this implies that $f^{-1}(N)$ is a perfect matching of $K_{2r+1}$, a contradiction once again. Therefore, $H$ contains a complete graph of odd order as a simple spanning subgraph.
\end{proof}

Let $r>1$ and let $K_{2r+1}'$ be the complete graph on $2r+1$ vertices minus an edge. Let $J_{2r}$ be the graph obtained by considering $r$ copies of $K_{2r+1}'$ such that all the vertices of degree $2r-1$ in these copies are adjacent to a new vertex $u$, resulting in a $2r$-regular simple graph. We refer to the vertex $u$ as the central vertex of $J_{2r}$, and the $r$ copies of $K_{2r+1}'$ are denoted by $R_{1}, \ldots, R_{r}$.

\begin{lemma}
Let $r>1$ and let $H$ be a graph such that $H\prec J_{2r}$. Then, $H\not\in\mathcal{K}_{t}^{2r}$, for all possible $t$.
\end{lemma}

\begin{proof}

Suppose that there exists a graph $H\in\mathcal{K}_{t}^{2r}$ such that $H\prec J_{2r}$, for contradiction. Let $f$ be an $H$-colouring of $J_{2r}$ and let $f_{V}$ be the induced map on the vertices of $J_{2r}$. 

Let $u_1,u_2$ be two vertices of $J_{2r}$ adjacent to $u$ and belonging to $R_1$ and $R_2$, respectively.
Consider a cycle $C$ of $H$ (possibly of length $2$) which contains the two edges $f(uu_1)$ and $f(uu_2)$ incident to $f_V(u)$. The preimage $f^{-1}(E(C))$ is a $2$-regular subgraph of $J_{2r}$  (as in the proof of Proposition \ref{prop:S12+kM}).
Moreover, one of the connected components of $f^{-1}(E(C))$ is a cycle passing through $u$ and containing the two edges $uu_1$ and $uu_2$, a contradiction since $J_{2r}$ does not have such a cycle.
\end{proof}

By the previous two lemmas, there exist no graph which colours both $K_{2r+1}$ and $J_{2r}$, implying our last result.

\begin{theorem}\label{theorem G simple even degree}
For every $r>1$, there is no connected graph $H$ colouring all $2r$-regular simple graphs.
\end{theorem}

Finally, as we have already remarked, we suggest the following open problem in order to have a complete answer to the general question asked in Section \ref{section regular}, that is,  whether there exists a connected graph $H$ such that for every $r$-regular graph $G$, $G$ admits an $H$-colouring, for each $r>3$. In order to fully answer this question, by Theorem \ref{theorem r_regular} and Theorem \ref{theorem G simple even degree}, it suffices to consider the following.

\begin{problem}\label{problem odd regular}
Let $r>1$ be odd. Determine whether there exists a connected graph $H$ colouring all $r$-regular simple graphs.
\end{problem}

The question whether there exists a graph $H$ in some class that colours any graph $G$ in some other class has been addressed in the cubic case considering various classes for both $H$ and $G$, for example, the class of bridgeless cubic graphs or the class of cubic graphs having a perfect matching. The same could be done in the case when $G$ is assumed to be an $r$-graph. Let us recall that an $r$-graph is a connected $r$-regular graph such that $|\partial{X}|\geq r$ for every odd subset $X$ of the vertex set. 
We thus suggest the following.

\begin{problem}
Let $r>3$. Determine whether there exists an $r$-graph $H$ colouring all (simple) $r$-graphs.
\end{problem}


\section{Acknowledgments}
The authors sincerely thank the anonymous referee for the precious and detailed suggestions, which led to significant improvements in the paper.

The first two authors were partially supported by the group GNSAGA of INdAM.

\end{document}